\documentclass[12pt, reqno]{amsart}
\usepackage{amssymb,latexsym,amsmath,amscd,amsthm,graphicx, color}
\usepackage{enumitem}
\usepackage[all]{xy}
\usepackage{pgf,tikz}
\usepackage{mathrsfs}
\usetikzlibrary{arrows}
\usepackage{pgfplots}
\pgfplotsset{compat=1.15}
\usepackage[left=0.6 in, top=0.6 in, right=0.6 in, bottom=0.3 in]{geometry}
\usepackage{changepage}
\allowdisplaybreaks
\makeatletter
\newcommand{\setword}[2]{%
	\phantomsection
	#1\def\@currentlabel{\unexpanded{#1}}\label{#2}%
}
\makeatother

\usepackage[linkcolor=blue, urlcolor=blue, citecolor=blue,
colorlinks, bookmarks]{hyperref}

\definecolor{uuuuuu}{rgb}{0.26666666666666666,0.26666666666666666,0.26666666666666666}
\definecolor{xdxdff}{rgb}{0.49019607843137253,0.49019607843137253,1.}
\definecolor{ffqqqq}{rgb}{1.,0.,0.}
\definecolor{ffqqqq}{rgb}{1.,0.,0.}
\definecolor{ffxfqq}{rgb}{1.,0.4980392156862745,0.}

\definecolor{uuuuuu}{rgb}{0.26666666666666666,0.26666666666666666,0.26666666666666666}
\definecolor{qqwuqq}{rgb}{0.,0.39215686274509803,0.}
\definecolor{zzttqq}{rgb}{0.6,0.2,0.}
\definecolor{xdxdff}{rgb}{0.49019607843137253,0.49019607843137253,1.}
\definecolor{qqqqff}{rgb}{0.,0.,1.}
\definecolor{cqcqcq}{rgb}{0.7529411764705882,0.7529411764705882,0.7529411764705882}
\definecolor{sqsqsq}{rgb}{0.12549019607843137,0.12549019607843137,0.12549019607843137}
\definecolor{uuuuuu}{rgb}{0.26666666666666666,0.26666666666666666,0.26666666666666666}
\definecolor{ffqqqq}{rgb}{1,0,0}
\definecolor{xdxdff}{rgb}{0.49019607843137253,0.49019607843137253,1}

\definecolor{yqqqyq}{rgb}{0.5019607843137255,0,0.5019607843137255}
\definecolor{qqqqff}{rgb}{0,0,1}
\definecolor{ffqqqq}{rgb}{1,0,0}
\definecolor{ffqqff}{rgb}{1,0,1}

\theoremstyle{plain}

\newtheorem{theorem}[subsection]{Theorem}
\newtheorem{theorem1}[subsubsection]{Theorem}

\newtheorem{lemma}[subsection]{Lemma}
\newtheorem{defi}[subsection]{Definition}

\newtheorem{proposition}[subsection]{Proposition}
\newtheorem{proposition1}[subsubsection]{Proposition}

\theoremstyle{definition}

\newtheorem{exam}[subsection]{Example}
\newtheorem{exam1}[subsubsection]{Example}

\newtheorem{remark}[subsection]{Remark}
\newtheorem{remark1}[subsubsection]{Remark}

\newcommand{\set}[1]{\{#1\}}

\newcommand{\ga}{\alpha}

\renewcommand{\gg}{\gamma}

\newcommand{\gq}{\theta}

\newcommand{\gt}{\tau}

\newcommand{\gG}{\Gamma}

\newcommand{\gQ}{\Theta}

\newcommand{\tbf}{\textbf}
\newcommand{\tit}{\textit}

\newcommand{\D}[1]{\mathbb{#1}}

\newcommand{\te}{\text}

\newcommand{\la}{\langle}
\newcommand{\ra}{\rangle}

\newcommand{\nd}{\noindent}

\begin{document}
 
 \nd To appear,\tit{ Mathematics} 
	\title{Optimal Quantization on Spherical Surfaces: \\
Continuous and Discrete Models -- A Beginner-Friendly Expository Study}
 
 \author{Mrinal Kanti Roychowdhury}
\address{Department of Mathematics and Statistical Science, University of Texas Rio Grande Valley, Edinburg, TX 78539, USA}
\email{mrinal.roychowdhury@utrgv.edu}

	
	\date{}
	\maketitle
	
	\pagestyle{myheadings}\markboth{Mrinal Kanti Roychowdhury}{Optimal Quantization on Spherical Surfaces}
 \begin{abstract}
 
 This expository paper provides a unified and pedagogical introduction to optimal
quantization for probability measures supported on spherical curves and discrete subsets of the sphere, emphasizing both
continuous and discrete settings. We first present a detailed geometric and analytical
foundation for intrinsic quantization on the unit sphere, including definitions of
great and small circles, spherical triangles, geodesic distance, Slerp interpolation,
the Fr\'echet mean, spherical Voronoi regions, centroid conditions, and quantization
dimensions.

Building upon this framework, we develop explicit continuous and discrete quantization
models on spherical curves, namely great circles, small circles, and great circular arcs—
supported by rigorous derivations and pedagogical exposition. For uniform continuous
distributions, we compute optimal sets of $n$-means and the associated quantization
errors on these curves; for discrete distributions, we analyze antipodal, equatorial,
tetrahedral, and finite uniform configurations, illustrating convergence to the
continuous model.

The central conclusion is that for a uniform probability distribution supported on a
one-dimensional geodesic subset of total length $L$, the optimal $n$-means form a
uniform partition and the quantization error satisfies $V_n = L^2/(12n^2)$. The
exposition emphasizes geometric intuition, detailed derivations, and clear
step-by-step reasoning, making it accessible to beginning graduate students and
researchers entering the study of quantization on manifolds. This article is intended
as an expository and tutorial contribution, with the main emphasis on geometric
reformulation and pedagogical clarity of intrinsic quantization on spherical curves,
rather than on the development of new asymptotic quantization theory.
 \end{abstract}

\setcounter{tocdepth}{2}
\tableofcontents

\section{Introduction and Geometric Preliminaries}

Quantization theory concerns the approximation of a probability distribution by a finite set of
representative points (or codepoints) in such a way that the expected distortion is minimized. The
classical foundations of Euclidean quantization were established through pioneering work of
Zador~\cite{Zador1982}, the extensive development by Gersho and Gray~\cite{GershoGray1992}, and the
authoritative survey of Gray and Neuhoff~\cite{GrayNeuhoff1998}. A rigorous measure-theoretic and
probabilistic framework for quantization of probability distributions was subsequently formulated by Graf
and Luschgy in their monograph~\cite{GrafLuschgy2000}. Statistical aspects such as consistency of the
$k$-means method were studied by Pollard~\cite{Pollard1982}, while learning-theoretic methods in vector
quantization were developed by Linder~\cite{Linder2002}. These works together form the basis of modern
Euclidean quantization theory.

In recent years, there has been growing interest in extending quantization to non-Euclidean and curved
spaces, particularly to probability measures supported on manifolds. On the sphere, quantization is
relevant in directional statistics~\cite{MardiaJupp2000}, geometric data analysis, and manifold-based
applications arising in computer vision, shape analysis, and machine learning. In this setting, classical
Euclidean notions such as straight lines, centroids, and Voronoi regions must be replaced by their intrinsic
geometric counterparts: geodesic arcs, Fréchet means~\cite{Frechet1948}, Karcher means~\cite{Karcher1977},
and spherical Voronoi tessellations.

While the general problem of optimal quantization on spherical surfaces is inherently
high-dimensional and analytically challenging, the focus of this paper is deliberately
more specific and pedagogical. We concentrate on \emph{intrinsic quantization problems
for probability measures supported on one-dimensional spherical curves}, namely great
circles, small circles, and geodesic arcs. In these settings, the sphere induces a
one-dimensional Riemannian structure via arc-length, allowing the intrinsic quantization
problem to reduce to a transparent and fully explicit one-dimensional model.

The central theme of the paper is that, once the correct intrinsic metric is adopted,
quantization on such spherical curves reduces to classical one-dimensional quantization
with respect to arc-length. This reduction explains why the optimal Voronoi cells are
contiguous intervals and why, under uniform density, the optimal codepoints are the
midpoints of these intervals.

A brief familiarity with differential geometry and intrinsic statistics on manifolds is helpful for
understanding quantization beyond Euclidean settings. For an accessible introduction to the geometry of
curves and surfaces, we refer the reader to Tapp~\cite{Tapp}, while the foundational framework of intrinsic
statistics on Riemannian manifolds and geometric measurements was developed in the influential work of
Pennec~\cite{Pennec2006}. These references provide essential background for readers wishing to deepen
their understanding of the geometric and statistical structures underlying quantization on curved spaces.

The aim of this introductory section is twofold. First, we provide a concise and pedagogical overview of
the geometric and analytical tools required for quantization on the sphere—geodesic distance, spherical
coordinates, Slerp interpolation, Fréchet means, Voronoi partitions, and centroid conditions. Second,
we establish the conceptual foundations that allow the reader to follow the developments in the later
sections smoothly, without requiring prior background in differential geometry. The exposition is
intentionally intuitive and example-driven, with the goal of making the subject accessible to beginning
graduate students and researchers in analysis, probability, or applied mathematics who wish to learn
quantization on manifolds for the first time.

Before entering the detailed geometric and analytic developments, we summarize
the guiding principle underlying the main results of this paper. While optimal
quantization on spherical surfaces is, in general, a high-dimensional and
analytically challenging problem, the focus of this paper is deliberately more
specific. We concentrate on intrinsic quantization problems for probability measures
supported on one-dimensional spherical curves, namely great circles, small circles,
and geodesic arcs. Although the ambient space is the two-dimensional sphere, once
the support of the probability measure is constrained to such a spherical curve,
the intrinsic quantization problem effectively reduces to a one-dimensional
Riemannian model with arc-length parameterization. This reduction principle allows
for a transparent and explicit analysis.

\medskip
\noindent\textbf{Nature and contribution of the paper.}
The results presented here recover, in the intrinsic spherical setting, the classical
one-dimensional uniform quantization formula
\[
V_n = \frac{L^2}{12n^2},
\]
where $L$ denotes the intrinsic length of the support curve. Accordingly, the primary
contribution of this paper is not the discovery of new quantization asymptotics, but
rather a geometric re-framing of known one-dimensional results within the intrinsic
geometry of the sphere. The paper is intended as a tutorial and pedagogical note,
highlighting how intrinsic metrics, geodesic Voronoi structures, and arc-length
parameterization lead naturally to transparent and explicit solutions on spherical
curves.

\medskip
\noindent
The reduction principle described above is summarized in the following conceptual
theorem.

\begin{theorem}[Conceptual reduction to the one-dimensional uniform case]
Let $C \subset \D S^2_\rho$ be a smooth closed geodesic curve of total intrinsic length $L$
(e.g.\ a great circle or a small circle), and let $P$ be the uniform probability measure
on $C$ with respect to arc-length. Then the intrinsic quantization problem on $(C,d_G)$
with squared geodesic distortion is equivalent to the classical one-dimensional uniform
quantization problem on a circle of length $L$.

In particular, for each $n \ge 1$, the optimal set of $n$-means consists of $n$ equally
spaced points along $C$, each Voronoi region has length $L/n$, each codepoint is the
geodesic midpoint of its cell, and the quantization error satisfies
\[
V_n = \frac{L^2}{12n^2}.
\]
\end{theorem}

The present article is intended as a beginner-friendly expository companion to the
author’s forthcoming comprehensive monograph, which will provide a deeper and more
systematic treatment of quantization on spherical surfaces~\cite{Roychowdhury2025}.
Accordingly, we emphasize geometric intuition, explicit derivations, and pedagogical
clarity throughout. This section provides the essential background for the study of
quantization on spherical surfaces, summarizing the geometric and analytical components
that will be used throughout the paper.

\subsection{Great Circle, Small Circle, and Spherical Triangle}
A \emph{great circle} on a sphere is the intersection of the sphere with a plane that passes through the \emph{center} of the sphere. 
Equivalently, it is a circle on the sphere whose center coincides with the center of the sphere. 
A \emph{small circle} on a sphere is the intersection of the sphere with a plane that does \emph{not} pass through the center of the sphere. 
The center of the small circle lies on the line connecting the center of the sphere and the point on the plane nearest to the center, 
but it does not coincide with the sphere’s center. 
  A \emph{spherical triangle} on the surface of a sphere is the region bounded by three arcs of great circles, 
each pair of which intersects at a vertex. 
The three vertices lie on the sphere, and the sides are segments of great circles connecting these vertices. 
The \emph{angles} of a spherical triangle are the dihedral angles between the planes of the great circles at their intersections.

 \subsection{Equator and the Prime Meridian, Latitude and Longitude.} 
The \emph{Equator} is an imaginary closed curve on the surface of the Earth that lies equidistant from the 
North and South Poles. Geometrically, it is a \emph{great circle} and it divides the Earth into the \emph{Northern Hemisphere} and the 
\emph{Southern Hemisphere}. Latitude measures how far north or south a point is from the Equator. It ranges from $0^{\circ}$ at the Equator to $90^{\circ}$ North (the North Pole) and $90^{\circ}$ South (the South Pole). 
 The \emph{Prime Meridian} is an imaginary semicircular great circle on the surface of the Earth that passes 
through the \emph{North Pole} and \emph{South Pole} and the \emph{Royal Observatory in Greenwich, England}. It divides the Earth into the Eastern Hemisphere and the Western Hemisphere. Longitude measures how far east or west a point is from the Prime Meridian. It ranges from $0^{\circ}$ at the Prime Meridian  to $180^{\circ}$ East and $180^{\circ}$ West. The Equator is the $0^{\circ}$ line of latitude, and the Prime Meridian is the $0^{\circ}$ line of Longitude. 
Latitude and Longitude form a coordinate pair: $(Latitude,Longitude)$. For example: New York City $\approx (40.71^{\circ} N, 74.00^{\circ} W)$, 
Rio de Janeiro $\approx(22.91^{\circ} S, 43.17^{\circ} W)$, London $\approx(51.51^{\circ} N, 0.13^{\circ} W)$.

 \subsection{Relationship between three different coordinates: Cartesian, Spherical, and Geographical} 
The relationship between Cartesian coordinates \((x, y, z)\) and spherical coordinates \((\rho, \theta, \phi)\) is given by the conversion formulas:
\begin{equation} \label{Pav1} 
x = \rho \sin \phi \cos \theta, \qquad
y = \rho \sin \phi \sin \theta, \qquad
z = \rho \cos \phi.
\end{equation} 
where 
\[
\begin{aligned}
\rho \ & : \ \text{the radial distance from the origin to the point,} \\
\theta \ & : \ \text{the azimuthal angle, measured in the } xy\text{-plane from the positive } x\text{-axis (longitude),} \\
\phi \ & : \ \text{the polar angle, measured from the positive } z\text{-axis (colatitude).}
\end{aligned}
\]
If the radius $\rho$ is fixed, the spherical coordinates of the point can be identified as $(\gq, \phi)$. Notice that if the radius $\rho$ of the sphere is fixed the same point in latitude-longitude coordinates is represented by $(\frac \pi 2-\phi, \gq)$. The latitude-longitude coordinates of a point on a sphere are called the \tit{geographical coordinates} of the point.  
Let the geographical coordinates of a point on the sphere be given by $(\phi, \gq)$, where 
\[
\phi \in \left[-\frac{\pi}{2}, \frac{\pi}{2}\right] \quad \text{(latitude, north positive)}, \te{ and }
\gq \in (-\pi, \pi] \quad \text{(longitude, east positive)}.
\]
Then, by \eqref{Pav1} we have the embedding of $(\phi, \gq)$ into $\mathbb{R}^3$ as the vector
\begin{equation} \label{Pav2}
\tbf x(\phi, \gq) = (x, y, z) = 
\bigl( \rho \cos \phi \cos \gq,\, \rho \cos \phi \sin \gq,\, \rho \sin \phi \bigr).
\end{equation}

\subsection{Geodesic Distance via Geographical Coordinates}
The geodesic distance between two points on a surface (like a sphere) is the shortest possible distance along the surface that connects them.
It is the analog of a ``straight line distance" in flat Euclidean space — but restricted to move on the surface.
On a sphere, the geodesics are great circle arcs, so the geodesic distance between two points on the sphere equals the length of the shorter great circle arc joining them.
Consider two points on a sphere of radius $\rho$ with geographical coordinates
\[
P_1 = (\phi_1, \gq_1), \qquad P_2 = (\phi_2, \gq_2).
\]
For two points $P_1$ and $P_2$, their corresponding vectors are
\[
\mathbf{x}_1 = \mathbf{x}(\phi_1,\gq_1),
\qquad
\mathbf{x}_2 = \mathbf{x}(\phi_2,\gq_2).
\]
 The dot product is
\begin{align*}
\mathbf{x}_1 \cdot \mathbf{x}_2
&= \rho^2\left[(\cos\phi_1\cos\gq_1)(\cos\phi_2\cos\gq_2)
+ (\cos\phi_1\sin\gq_1)(\cos\phi_2\sin\gq_2)
+ (\sin\phi_1)(\sin\phi_2)\right]\\[0.5em]
&=\rho^2\left[ \cos\phi_1\cos\phi_2\bigl(\cos\gq_1\cos\gq_2 + \sin\gq_1\sin\gq_2 \bigr)
+ \sin\phi_1\sin\phi_2\right]\\[0.5em]
&= \rho^2\left[\cos\phi_1\cos\phi_2\cos(\gq_1 - \gq_2) + \sin\phi_1\sin\phi_2\right].
\end{align*}
Therefore, the \emph{central angle} $\gQ$ between $P_1$ and $P_2$ satisfies
\[
 \gQ = \cos^{-1}\Big( \frac{x_1\cdot x_2}{\rho^2}\Big)=\cos^{-1}\Big (\sin\phi_1\sin\phi_2 + 
\cos\phi_1\cos\phi_2\cos(\gq_1 - \gq_2)\Big).
\]
Then,  
\begin{equation*}  
d_G(P_1, P_2) = \rho  \gQ=\rho\cos^{-1}\Big (\sin\phi_1\sin\phi_2 + 
\cos\phi_1\cos\phi_2\cos(\gq_1 - \gq_2)\Big),
\end{equation*} 
which is known as the \tit{geodesic distance between $P_1$ and $P_2$} via geographical coordinates on a sphere of radius~$\rho$.

\subsection{Arc Length on a Spherical Surface}
 Let 
\[
\D S^2_\rho = \{ (x,y,z) \in \mathbb{\D R}^3 : x^2 + y^2 + z^2 = \rho^2 \}
\]
be a sphere of radius \( \rho > 0 \).  
Let \( \Gamma \subset \D S^2_\rho \) be a smooth curve lying on the spherical surface.
Suppose that \( \Gamma \) admits a smooth parametrization
\[
\boldsymbol{\gamma} : [a,b] \to \D S^2_\rho, \qquad 
\boldsymbol{\gamma}(t) = (x(t), y(t), z(t)),
\]
such that \( \|\boldsymbol{\gamma}(t)\| = \rho \) and \( \boldsymbol{\gamma}'(t) \neq 0 \) for all \( t \in [a,b] \).
The \emph{arclength element} $ds$ along the curve is
\[
ds = \|\boldsymbol{\gamma}'(t)\| \, dt.
\]
$ds$ is also known as the differential of the arclength. 
The total length of the curve is given by
\[
L(\Gamma) = \int_a^b  ds=\int_a^b \|\boldsymbol{\gamma}'(t)\| \, dt.
\]

\subsection{The two-argument arctangent function
\texorpdfstring{$\operatorname{atan2}(y,x)$}{atan2}}

To represent angles in the full range $(-\pi,\pi]$ without ambiguity, we use the
two-argument arctangent function $\theta=\operatorname{atan2}(y,x)$. For any
$(x,y)\neq(0,0)$, this function returns the unique angle $\theta\in(-\pi,\pi]$ such
that
\[
\cos\theta=\frac{x}{\sqrt{x^2+y^2}},\qquad
\sin\theta=\frac{y}{\sqrt{x^2+y^2}}.
\]
Unlike the single-variable function $\arctan(y/x)$, the function
$\operatorname{atan2}(y,x)$ automatically accounts for the correct quadrant of the
point $(x,y)$ and is standard in geometry, navigation, and numerical implementations.

In this paper, $\operatorname{atan2}$ is used only to define longitude coordinates
consistently and plays no further conceptual role in the quantization analysis.

\subsection{Spherical Linear Interpolation (Slerp)}
It is a smooth parametrization of the shortest geodesic (great-circle) arc connecting two points on a sphere. 

Let $\D S^2_\rho = \{\, x \in \mathbb{R}^3 : \|x\| = \rho \,\}$ be a sphere of radius $\rho > 0$, 
and let $u_A, u_B \in\D S^2_\rho$ be two distinct points. 
Denote by 
\[
s = \arccos\!\left( \frac{\langle u_A, u_B \rangle}{\rho^2} \right)
\]
the central angle (in radians) subtended by $u_A$ and $u_B$ at the center of the sphere, where $\langle u_A, u_B \rangle$ is the Euclidean inner product (dot product) between the 3D vectors $u_A$ and $u_B$. 
Then, the \tit{Slerp curve} between $u_A$ and $u_B$ is defined by
\[
\gamma_{AB}(\tau)
= \frac{\sin((1-\tau)s)}{\sin s}\,u_A
+ \frac{\sin(\tau s)}{\sin s}\,u_B,
\qquad \tau \in [0,1].
\]
Obviously, the curve satisfies $\gamma_{AB}(0)=u_A$, $\gamma_{AB}(1)=u_B$. Moreover, $\|\gamma_{AB}(\tau)\|=\rho$ and  $\|\gamma_{AB}'(\tau)\|=\rho s$, which is a constant (see Proposition~\ref{prop00})  for all $\tau \in [0,1]$. Thus, the curve lies entirely on the sphere with length the geodesic distance $d_G(u_A, u_B)=\rho s$ between $u_A$ and $u_B$. $\gg_{AB}(\gt)$ traces the unique great-circle arc connecting $u_A$ and $u_B$. As $\gt$ increases uniformly from $0$ to $1$, the central angle from $u_A$ to $\gamma_{AB}(\tau)$ increases linearly from $0$ to $s$; hence, the motion along the arc has constant angular speed and covers equal arc lengths for equal increments of $\gt$.

\begin{proposition1} \label{prop00} 
For the Slerp curve
\[
\gamma_{AB}(\tau)
= \frac{\sin((1-\tau)s)}{\sin s}\,u_A
+ \frac{\sin(\tau s)}{\sin s}\,u_B,
\qquad \tau \in [0,1],
\]
where $\|u_A\|=\|u_B\|=\rho$ and $\langle u_A,u_B\rangle=\rho^2\cos s$,
we have
\[
\|\gamma_{AB}(\tau)\| = \rho, \te{ and } \|\gg_{AB}'(\gt)\|=\rho s.
\]
\end{proposition1} 

\begin{proof}
We have 
\[
\|\gamma_{AB}(\tau)\|^2 = \langle \gamma_{AB}(\tau),\, \gamma_{AB}(\tau) \rangle=\langle a\,u_A + b\,u_B, a\,u_A + b\,u_B\rangle,
\]
where
\[
a = \frac{\sin((1-\tau)s)}{\sin s}  \te{ and } 
b = \frac{\sin(\tau s)}{\sin s}.
\]
Then
\[
\|\gamma_{AB}(\tau)\|^2 
= a^2 \|u_A\|^2 
+ b^2 \|u_B\|^2 
+ 2ab\,\langle u_A, u_B \rangle.
\]
Because $u_A, u_B\in S_\rho^2$, we have 
\[\|u_A\|=\|u_B\|=\rho \te{ and } \la u_A, u_B\ra=\rho^2\cos s\]
implying \[
\|\gamma_{AB}(\tau)\|^2 
= \rho^2(a^2 +b^2+ 2ab\cos s).
\]
Substituting the values of $a$ and $b$, we have 
\[
a^2 + b^2 + 2ab\cos s 
= 
\frac{
\sin^2((1-\tau)s)
+ \sin^2(\tau s)
+ 2\sin((1-\tau)s)\sin(\tau s)\cos s
}{\sin^2 s}.
\]
Using the product-to-sum identity
\[
\sin X \sin Y 
= \tfrac{1}{2}\big(\cos(X-Y) - \cos(X+Y)\big),
\]
and after simplification, we obtain
\[
\sin^2((1-\tau)s)
+ \sin^2(\tau s)
+ 2\sin((1-\tau)s)\sin(\tau s)\cos s
= \sin^2 s.
\]
Thus,
\[
a^2 + b^2 + 2ab\cos s
= \frac{\sin^2 s}{\sin^2 s}
= 1 \te{ yielding } \|\gamma_{AB}(\tau)\|^2
= \rho^2, \te{ i.e., }    \|\gamma_{AB}(\tau)\|=\rho. 
\]
To show $\|\gg_{AB}'\|=\rho s$, we proceed as follows:

Let \[u_A=\tbf x(\phi_1, \gq_1), \te{ and } u_B=\tbf x(\phi_2, \gq_2).\]
 Write the associated unit vectors: 
\[\hat u_A=\frac {u_A} \rho, \te{ and } \hat u_B=\frac {u_B}\rho.\]
Their central angle $s \in [0, \pi]$ is independent of $\rho$, i.e., $\hat u_A \cdot \hat u_B = \cos s, \text{ where } s \in [0,\pi]$. Build an orthonormal basis of the plane span$\set{\hat u_A, \hat u_B}$: 
\begin{equation} \label{eqPav21} e_1:=\hat u_A, \te{ and } e_2:=\frac{\hat u_B-(\hat u_A\cdot \hat u_B) \hat u_A}{\|\hat u_B-(\hat u_A\cdot \hat u_B) \hat u_A\|}.\end{equation} 
Write
\[
w := \hat u_B - (\hat u_A \cdot \hat u_B)\hat u_A = \hat u_B - \cos s\,\hat u_A.
\]
This is the component of \(\hat u_B\) orthogonal to \(\hat u_A\).
We compute its norm:
\[
\begin{aligned}
\|w\|^2 
&= \|\hat u_B - \cos s\,\hat u_A\|^2 \\[4pt]
&= \|\hat u_B\|^2 - 2\cos s\,(\hat u_A \cdot \hat u_B) + \cos^2 s\,\|\hat u_A\|^2 \\[4pt]
&= 1 - 2\cos^2s + \cos^2s \qquad  (\text{since }\|\hat u_A\|=\|\hat u_B\|=1) \\[4pt]
&= 1 - \cos^2s \\[4pt]
&= \sin^2 s.
\end{aligned}
\]
Therefore, 
\[
\|\,\hat u_B - (\hat u_A \cdot \hat u_B)\hat u_A\,\| = \sin s.
\]
Hence, by \eqref{eqPav21}, we have 
\[
e_2 = \frac{\hat u_B - (\hat u_A \cdot \hat u_B)\hat u_A}{\|\hat u_B - (\hat u_A \cdot \hat u_B)\hat u_A\|}
     = \frac{\hat u_B - \cos s\,\hat u_A}{\sin s} \te{ implying } \hat u_B=\cos s \, e_1+\sin s \,e_2.
\]
Hence, for $\gt \in [0, 1]$, we have 
\begin{align*}
\gamma_{AB}(\tau)
&= \frac{\sin((1-\tau)s)}{\sin s}\,\rho \hat u_A
+ \frac{\sin(\tau s)}{\sin s}\,\rho \hat u_B =\rho\Big(\frac{\sin((1-\tau)s)}{\sin s}\, e_1
+ \frac{\sin(\tau s)}{\sin s}(\cos s \, e_1+\sin s \,e_2)\Big) \\
&=\rho\Big(\cos(\gt s) e_1+ \sin (\gt s) e_2\Big).
\end{align*}
Differentiating with respect to $\gt$, we have 
\[\gamma'_{AB}(\tau)=\rho s\Big(-\sin(\gt s) e_1+ \cos (\gt s) e_2\Big).\]
Since $e_1$ and $e_2$ are orthonormal, we have $\|\gg_{AB}'(\gt)\|=\rho s$. Thus, the proof is complete. 
 \end{proof}
\begin{remark1} In Proposition~\ref{prop00}, $s$ is the angle between the vectors $u_A, u_B\in \D S^2$. If $s=0$, then $u_A=u_B$, and hence we can take the Slerp curve as a constant curve $\gg_{AB}(\gt)=\rho u_A$ for $0\leq \gt \leq 1$. Then, $\gg_{AB}'(\gt)=0$ for all $0\leq \gt\leq 1$. Hence, the length of the Slerp curve is 
\[L=\int_0^1 \|\gg'_{AB}(\gt)\|\,d\gt=0=\rho  0=\rho s.\]
Likewise, the speed is $\|\gg'_{AB}(\gt)\|=0=\rho s$. On the other hand, if $s = \pi$, then the vectors $u_A$ and $u_B$ are antiparallel, i.e., $u_B = -u_A$. The short geodesic is any great semicircle joining $u_A$ to $-u_A$;
it is not unique because there are infinitely many planes through the origin containing
the line $\mathbb{R}u_A=\set{tu_A : t\in \D R}$.
Choose any unit vector $e_2 \perp u_A$. Then the great semicircle can be parameterized as
\[
\gg_{AB}(\tau) = \rho\big(\cos(\pi\tau)\,u_A + \sin(\pi\tau)\,e_2\big),
\qquad 0 \le \tau \le 1.
\]
This curve starts at $A = \rho u_A$ and ends at $B = \rho u_B = -\rho u_A$.
Differentiating with respect to $\tau$, we obtain
\[
\gg_{AB}'(\tau)
= \rho\pi\big(-\sin(\pi\tau)\,u_A + \cos(\pi\tau)\,e_2\big).
\]
Because $u_A$ and $e_2$ are orthonormal, the speed is constant:
\[
\|\gg_{AB}'(\tau)\|
= \rho\pi
= \rho s.
\]
Hence the total length of the semicircular geodesic is
\[
L = \int_0^1 \|\gg_{AB}'(\tau)\|\,d\tau
  = \int_0^1 \rho\pi\,d\tau
  = \rho\pi
  = \rho s.
\]
\end{remark1}

\subsection{Fr\'echet Mean}

Let $(M,d)$ be a metric space and let $P$ be a Borel probability measure on $M$.  
The \emph{Fr\'echet mean} (also called the \emph{intrinsic mean} or \emph{Riemannian center of mass}) of $P$ is defined as the point or set of points in $M$ minimizing the expected squared distance to $P$. Formally,
\begin{equation}\label{eq:frechet}
\mu^{*} = \arg\min_{q\in M} \int_{M} d^{2}(x,q)\,dP(x).
\end{equation}
When the minimizer is unique, $\mu^{*}$ is called \emph{the} Fr\'echet mean of $P$; otherwise, the set of all minimizers is referred to as the \emph{Fr\'echet mean set} of $P$.  
The function 
\[
F(q) := \int_{M} d^{2}(x,q)\,dP(x)
\]
is known as the \emph{Fr\'echet functional}. A point $\mu^{*}$ satisfying \eqref{eq:frechet} is the unique global minimizer of $F$ whenever $F$ is strictly convex.

\vspace{1em}
\noindent\textbf{Interpretation.}
The Fr\'echet mean generalizes the classical Euclidean mean to arbitrary metric or Riemannian spaces.  
In Euclidean space, the minimizer of the mean squared distance coincides with the ordinary arithmetic mean.  
On a curved manifold, the distance $d$ is replaced by the geodesic distance $d_{G}$, so that the Fr\'echet mean provides the natural notion of ``average'' consistent with the intrinsic geometry of $M$.

\begin{exam1}[Euclidean space]
Let $M=\mathbb{R}^{n}$ with the standard Euclidean distance $d(x,q)=\|x-q\|$.  
For a random variable $X$ with distribution $P$, the Fr\'echet functional becomes
\[
F(q) = \int_{\mathbb{R}^{n}} \|x-q\|^{2}\,dP(x).
\]
Differentiating with respect to $q$ and setting $\nabla F(q)=0$ yields
\[
q = \int_{\mathbb{R}^{n}} x\,dP(x) = \mathbb{E}[X].
\]
Hence the Fr\'echet mean coincides with the classical Euclidean (arithmetic) mean.
\end{exam1}

\begin{exam1}[Spherical space]
Let $M=\D S^2$ be the unit sphere in $\mathbb{R}^{3}$ equipped with the geodesic distance
\[
d_{G}(x,q)=\arccos(\langle x,q\rangle),
\]
where $\langle \cdot,\cdot\rangle$ denotes the standard inner product in $\mathbb{R}^{3}$.  
For a probability distribution $P$ on $\D S^2$, the Fr\'echet mean minimizes
\[
F(q)=\int_{\D S^2} d_{G}^{2}(x,q)\,dP(x)
     =\int_{\D S^2} \arccos^{2}(\langle x,q\rangle)\,dP(x).
\]
If $P$ is the uniform distribution on a geodesic arc of the sphere, the unique minimizer $\mu^{*}$ is the midpoint of the arc in geodesic distance.  
Similarly, for a uniform distribution on a symmetric closed curve such as the boundary of a spherical triangle, the Fr\'echet mean lies on the axis of symmetry of that curve.
\end{exam1}

\subsection{Quantization Error on the Sphere}
Let $P$ be a Borel probability measure on a sphere $\D S_\rho^2$ of radius $\rho$ equipped with the geodesic metric $d_G$. 
Let \( \alpha = \{ a_1, a_2, \dots, a_n \} \subset \D S_\rho^2 \) be a finite set of points.  Such a set $\ga$ is also referred to as codebook and the elements as codepoints. Let $r \in\D R$ with $r>0$. 
Then, the \emph{distortion error} for $P$, of order \( r > 0 \) associated with \( \alpha \), denoted by $V_r(P; \ga)$, is defined by
\[
V_{r}(P; \alpha): = \int_{\D S_\rho^2} \min_{a_i \in \alpha} d^r_G(x, a_i) \, dP(x). 
\]
 Write
\[
V_{n, r}(P):= \inf_{\substack{\alpha \subset \D S_\rho^2 \\ |\alpha| \le n}} V_{r}(P; \alpha).
\]
$V_{n, r}(P)$ is called the \tit{$n$th quantization error of order $r$} for the probability measure $P$. 
When \( r = 2 \), the problem corresponds to minimizing the mean squared geodesic distance between a random vector $X$ with distribution $P$ and its nearest codepoint. Then, $V_{n, r}(P)$ is  called the $n$th quantization error for $P$ with respect to the squared geodesic distance.  

\subsection{Spherical Voronoi Regions}

Given a codebook \( \alpha = \{ a_1, \dots, a_n \} \subset \D S_\rho^2 \), the sphere can be partitioned into \emph{spherical Voronoi regions}
\[
R(a_i|\ga)= \{ x \in \D S_\rho^2 : d_G(x, a_i) \le d_G(x, a_j), \ \te{ for all } 1\leq j \leq n \te{ and } j \neq i  \}.
\]
Each region \( R(a_i|\ga)\) consists of all points on the sphere closer to \( a_i \) than to any other codepoint (with respect to the geodesic metric).
The distortion error $V_{r}(P; \ga)$ can then be written as
\[
V_{r}(P; \alpha) = \sum_{i=1}^{n} \int_{R(a_i|\ga)} d^r_G(x, a_i) \, dP(x).
\]
 
\subsection{Optimal Quantizers on the Sphere}

A codebook \( \alpha^* = \{ a_1^*, \dots, a_n^* \} \subset \D S_\rho^2 \) is called an \tit{optimal codebook}, also called an \tit{optimal set of $n$-means}, if
\[
V_{n, r}(P) = V_{r}(P; \alpha^*).
\]
Each element of an optimal codebook is called an \tit{optimal codepoint} or \tit{optimal quantizer}. 
In the case \( r = 2 \), each optimal codepoint satisfies a \emph{spherical centroid condition} analogous to the Euclidean case.

\subsection{Extrinsic centroid condition and its relation to intrinsic (Fr\'echet) means}
\label{subsec:centroid-extrinsic-intrinsic}

Let $P$ be a Borel probability measure supported on the sphere
\[
\D S^2_\rho := \{x \in \mathbb{R}^3 : \|x\| = \rho\},
\]
equipped with the geodesic distance $d_G$. Let
\[
\alpha^*=\{a_1^*,a_2^*,\dots,a_n^*\}\subset\D S^2_\rho
\]
be an optimal set of $n$-means for $P$ (with respect to squared geodesic distortion), and let
$R(a_i^*\mid \alpha^*)$ denote the associated spherical Voronoi regions.

\medskip
\noindent\textbf{Extrinsic centroid condition.}
In many applied and algorithmic treatments of quantization on spheres (notably spherical $k$-means),
one considers an \emph{extrinsic} formulation obtained by viewing $\D S^2_\rho$ as embedded in
$\mathbb{R}^3$ and minimizing Euclidean squared distance subject to the spherical constraint.
More precisely, for each Voronoi region $R(a_i^*\mid \alpha^*)$ one minimizes
\[
\int_{R(a_i^*\mid \alpha^*)}\|x-a\|^2\,dP(x)
\quad \text{subject to } \|a\|=\rho.
\]
Define the (unnormalized) Euclidean conditional expectation over the Voronoi region by
\[
m_i := \int_{R(a_i^*\mid \alpha^*)} x\,dP(x).
\]
Since typically $\|m_i\|<\rho$, the constrained minimizer is obtained by radial projection of $m_i$
onto $\D S^2_\rho$, yielding the \emph{extrinsic centroid}
\begin{equation}   \label{eq5}
a_i^* \;=\; \rho\,\frac{m_i}{\|m_i\|}
\;=\;
\rho\,
\frac{\displaystyle \int_{R(a_i^*\mid \alpha^*)} x\,dP(x)}
{\displaystyle \left\|\int_{R(a_i^*\mid \alpha^*)} x\,dP(x)\right\|}.
\end{equation}
We emphasize that \eqref{eq5} is an \emph{extrinsic} (embedded/chordal) centroid rule: it arises from
Euclidean squared loss in $\mathbb{R}^3$ together with the constraint $\|a\|=\rho$.

\medskip
\noindent\textbf{Relation to intrinsic (Fr\'echet) means.}
The intrinsic (Fr\'echet) mean on $\D S^2_\rho$ is defined as a minimizer of the geodesic Fr\'echet functional
\[
q \;\mapsto\; \int_{\D S^2_\rho} d_G(x,q)^2\,dP(x).
\]
In general, the extrinsic centroid \eqref{eq5} does \emph{not} coincide with the intrinsic Fr\'echet mean,
since the intrinsic problem depends nonlinearly on the Riemannian geometry through the exponential and
logarithm maps.

For a Voronoi region $R\subset \D S^2_\rho$, the correct intrinsic first-order optimality condition for a
minimizer $a\in\D S^2_\rho$ of
\[
F_R(a):=\int_R d_G(x,a)^2\,dP(x)
\]
is the vanishing of the Riemannian gradient, which can be expressed (when $R$ lies in a domain where the
logarithm map is well-defined) as 
\begin{equation} \label{eq:intrinsic-stationarity} 
\int_R \log_a(x)\,dP(x)=0,
\end{equation}
where $\log_a(\cdot)$ denotes the Riemannian logarithm map at $a$ on the sphere.
In general, \eqref{eq:intrinsic-stationarity} is an implicit condition and does not yield a closed-form
centroid formula.

\medskip
\noindent\textbf{Coincidence in the one-dimensional uniform models studied here.}
Although intrinsic and extrinsic centroids differ in general, they coincide in several important
situations, including:
\begin{itemize}
\item when the support of the conditional distribution $P(\cdot\mid R(a_i^*\mid \alpha^*))$ is contained
in a sufficiently small geodesic ball (so the intrinsic Fr\'echet functional is strictly convex), and
\item when the conditional distribution is symmetric with respect to a geodesic midpoint.
\end{itemize}
In particular, for the one-dimensional uniform models treated in Sections~4--7, each Voronoi region is a
geodesic interval (arc), and the intrinsic Fr\'echet mean reduces to the geodesic midpoint of that
interval. Consequently, in these settings the representative points produced by the extrinsic centroid
rule \eqref{eq5} agree with the intrinsic minimizers for squared geodesic distortion.

\medskip
\noindent\textbf{Remark.}
A fully intrinsic ``centroid condition'' for optimal quantizers on $\D S^2_\rho$ can be developed by
analyzing the Riemannian gradient of the Fr\'echet functional and using the logarithm map, leading to
\eqref{eq:intrinsic-stationarity}. Since the focus of this paper is on explicit and pedagogical models
where intrinsic means are available in closed form (and coincide with midpoints on geodesic intervals),
we restrict attention to the clarification above.

\subsection{Quantization Dimension and Quantization Coefficient}
Let $P$ be a Borel probability measure on a sphere $\D S_\rho^2$ of radius $\rho$, equipped with the geodesic metric $d_G$. 
Let $V_{n,r}(P)$ denote the $n$th quantization error of order $r>0$ for $P$. 
If the following limit exists,
\[
D_r(P)
    = \lim_{n \to \infty}
      \frac{r \log n}{-\log V_{n,r}(P)},
\]
then $D_r(P)$ is called the \emph{quantization dimension of order $r$} of the measure $P$. 
It measures the asymptotic rate at which the optimal quantization error $V_{n,r}(P)$ decreases as $n$ increases.
In particular, if
\[
V_{n,r}(P) \asymp C n^{-r/s}
\quad \text{for some constant } C>0,
\]
then $D_r(P) = s$.
Assuming that $D_r(P) = s$ exists, the \emph{upper} and \emph{lower quantization coefficients of order $r$} are defined respectively by
\[
\overline{Q}_{r}^{\,s}(P)
    = \limsup_{n \to \infty} n^{r/s} V_{n,r}(P),
\qquad
\underline{Q}_{r}^{\,s}(P)
    = \liminf_{n \to \infty} n^{r/s} V_{n,r}(P).
\]
If both limits coincide, i.e.,
\[
\overline{Q}_{r}^{\,s}(P)
    = \underline{Q}_{r}^{\,s}(P)
    =: Q_{r}^{\,s}(P),
\]
then $Q_{r}^{\,s}(P)$ is called the \emph{$s$-dimensional quantization coefficient of order $r$} for the measure $P$.
\medskip

\nd \textbf{Interpretation.}
The quantization coefficient provides the asymptotic constant in the rate of decay of the quantization error:
\[
V_{n,r}(P) \sim Q_r^s(P)\, n^{-r/s} \quad \text{as } n \to \infty.
\]
Hence, $D_r(P)$ characterizes the scaling exponent, while $Q_r^s(P)$ gives the precise asymptotic constant depending on the geometry of the support of $P$ on $\D S_\rho^2$. 
 
\subsection{Applications}

Quantization on spheres has applications in various areas such as:
\begin{itemize}
  \item Directional statistics,  and earth and planetary sciences (e.g., wind directions, orientations).
  \item Quantization of probability measures on compact manifolds.
  \item Spherical coding and communication systems.
  \item Computer graphics and spherical data compression.
\end{itemize}

\section{Geometry of the Unit Sphere}

\medskip
\noindent\fbox{%
\parbox{0.96\linewidth}{%
\textbf{Coordinate conventions used throughout the paper.}
We work on the sphere
$S^2_\rho = \{x \in \mathbb{R}^3 : \|x\| = \rho\}$ using standard spherical coordinates
$(\rho,\theta,\phi)$ defined by
\[
x = \rho(\sin\phi\cos\theta,\;\sin\phi\sin\theta,\;\cos\phi),
\]
where
\[
\rho > 0,\qquad \theta \in [0,2\pi),\qquad \phi \in [0,\pi].
\]
Here $\phi$ denotes the \emph{colatitude} (measured from the north pole), while
$\theta$ denotes the longitude.

For readers familiar with geographical coordinates, the latitude $\lambda \in
[-\pi/2,\pi/2]$ is related to $\phi$ by
\[
\lambda = \frac{\pi}{2} - \phi.
\]
Throughout the paper, we consistently use the colatitude $\phi$ rather than
latitude $\lambda$, unless explicitly stated otherwise.
}}
\medskip

\medskip 
\noindent\textbf{Normalization to the unit sphere.}
Throughout this section we work on the unit sphere $\D S^{2} \subset \mathbb{R}^{3}$.
This choice is made purely for notational convenience and involves no loss of
generality. Indeed, for a sphere $\D S^{2}_{\rho}$ of arbitrary radius $\rho>0$, the
intrinsic (geodesic) distance and arc-length scale linearly with $\rho$, i.e.,
$d_{G,\rho}(x,y)=\rho\, d_{G,1}(x/\rho,y/\rho)$. Consequently, any geodesic curve
of intrinsic length $L$ on $\D S^{2}_{\rho}$ corresponds, after normalization, to a
curve of length $L/\rho$ on the unit sphere. Since the quantization problems
studied in this paper depend only on the intrinsic one-dimensional geometry
through arc-length, all results obtained on $\D S^{2}$ extend immediately to
$\D S^{2}_{\rho}$ by the simple rescaling $L \mapsto \rho L$. Working on the unit
sphere therefore allows for a cleaner exposition while preserving full
generality.
\subsection{Coordinates and metric}
The unit sphere in $\mathbb R^3$ is
\[
\D S^2=\{(x,y,z)\in\mathbb R^3:x^2+y^2+z^2=1\}.
\]
In spherical coordinates, any point on the sphere is represented by $x(\theta,\phi)$, where 
\[
x(\theta,\phi)=(\cos\theta\sin\phi,\ \sin\theta\sin\phi,\ \cos\phi),
\]
where $\theta\in[0,2\pi)$ is the longitude and $\phi\in[0,\pi]$ is the colatitude of the point.  On the other hand, in geodesic coordinates, the same point on the sphere is represented by $x(\phi,\gq)$, where 
\[
x(\phi, \gq)=(\cos\phi\cos\theta,\ \cos\phi\sin\theta,\ \sin\phi),
\]
where $\phi \in[0,\pi]$ is the colatitude and $\theta\in[0,2\pi)$ is the longitude of the point. 

The intrinsic or \emph{geodesic distance} is the central angle between two points:
\[
d_G(x,y)=\arccos\langle x,y\rangle.
\]
When $x,y$ lie on the same great circle, $d_G(x,y)$ is simply their minimal angular separation. When no ambiguity arises, we write $d_G$ for the geodesic distance on the underlying manifold.

\subsection{Great and small circles}
\begin{itemize}
\item The \textbf{equator} (a great circle) is $\Gamma=\{(\cos\gq,\sin\gq,0):0\le\gq<2\pi\}$.
\item A \textbf{small circle} at latitude $\lambda$ is
\[
C_\lambda=\{(\cos\lambda\cos\gq,\cos\lambda\sin\gq,\sin\lambda):0\le\gq<2\pi\},
\]
of intrinsic length $L_\lambda=2\pi\cos\lambda$.
\item A \textbf{great-circle arc} is a connected subset of a great circle of geodesic length $L\in(0,2\pi]$.
\end{itemize}

These are the one--dimensional manifolds on which our continuous models are supported.
 \section{Quantization Framework}
Recall that for a metric space $(M,d)$ and a probability measure $P$ on $M$, the distortion of a finite codebook $\alpha=\{a_1,\dots,a_n\}$ is
\[
V(\alpha;P)=\int_M \min_{a\in\alpha}d(x,a)^2\,dP(x), 
\]
the $n$th quantization error is $V_n(P)=\inf_{|\alpha|\le n}V(\alpha;P)$.
Given $\alpha$, the \emph{Voronoi region} of $a_j$ is
\[
R_j=\{x\in M:d(x,a_j)\le d(x,a_k)\text{ for all }k\}.
\]

\begin{proposition}[Centroid condition; heuristic form]\label{prop:centroid}
Assume $M$ is a smooth curve with arc-length parameter $s$ and geodesic distance $d_G$, and $P$ has a continuous density $f$ with respect to $ds$. If $\alpha^\ast=\{a_i^\ast\}$ is optimal and $R_i$ is the Voronoi region of $a_i^\ast$, then $a_i^\ast$ minimizes
\[
F_i(a)\;=\; \int_{R_i} d_G(x,a)^2\, f(x)\, ds
\]
with respect to $a$ lying on the curve. In particular, when $R_i$ is a geodesic interval and $d_G$ coincides with interval length, $a_i^\ast$ is the \emph{midpoint} of $R_i$ when $f$ is constant.
\end{proposition}
\emph{Idea of proof.}
For variations of $a$ constrained to lie on the curve $M$, write the local functional
\[
F_i(a)=\int_{R_i} d_G(x,a)^2\,f(x)\,ds,
\]
and differentiate under the integral sign along any smooth curve $a(t)$ in $M$ with $a(0)=a$.
When $R_i$ is a geodesic interval and $d_G$ coincides with arc–length, set $s$ as the arc–length
coordinate on $R_i=[s_L,s_R]$ and let $s_0$ be the coordinate of $a$; then
\[
F_i(a)=\int_{s_L}^{s_R} (s-s_0)^2\,f(s)\,ds,\qquad
\frac{d}{ds_0}F_i(a) = 2\!\int_{s_L}^{s_R} (s_0-s)\,f(s)\,ds.
\]
Hence the first variation vanishes exactly when
\[
\int_{s_L}^{s_R} (s-s_0)\,f(s)\,ds=0,
\]
i.e., $s_0$ is the intrinsic (conditional) mean of $s$ over the cell. In particular, if $f$ is constant,
symmetry forces $s_0=(s_L+s_R)/2$, so $a$ is the midpoint of $R_i$. The second variation satisfies
\[
\frac{d^2}{ds_0^2}F_i(a)=2\int_{s_L}^{s_R}f(s)\,ds>0,
\]
so the critical point is a strict minimizer. A fully rigorous proof phrases the same calculation using the
Riemannian gradient of $x\mapsto d_G(x,a)^2$ and convexity along geodesic segments of $M$, which
justifies differentiating under the integral and yields uniqueness on each (short) geodesic interval. \qed

\begin{remark}
This is the spherical analogue of the Euclidean centroid rule: compute the weighted mean of each cell and project it to the sphere.  It forms the basis of spherical $k$--means algorithms.
\end{remark}

For Euclidean quantization, if $\alpha^*=\{a_1^*,\dots,a_n^*\}$ is a spherical optimal set of $n$-means, each $a_j^*$ is the conditional expectation of a random variable $X$ with distribution $P$ over its cell. On the sphere, the condition involves normalization, i.e., 
\begin{equation}\label{eq:centroid}
a_j^* = 
\frac{\sum_{x_i\in R_j} p_i x_i}
{\left\|\sum_{x_i\in R_j} p_i x_i\right\|}.
\end{equation}

\begin{remark}
Equation \eqref{eq:centroid} states: compute the weighted Euclidean mean of the assigned points
and project it back to the sphere by normalization. If $P$ is a continuous Borel probability measure, then the result follows analogously, see \eqref{eq5}. This is the discrete analogue of the extrinsic centroid condition discussed in Section~\ref{subsec:centroid-extrinsic-intrinsic}.
\end{remark}

\section{Uniform quantization on a geodesic circle: the equator}

\subsection{Setup and intuition}
Let $P$ be the uniform probability distribution on the equator $\Gamma$ with respect to arc-length. Intrinsically, $\Gamma$ is a \emph{circle of length $2\pi$} with the metric
\[
d_\Gamma\big(\gq_1,\gq_2\big)\;=\;\min\{|\gq_1-\gq_2|,\ 2\pi-|\gq_1-\gq_2|\}.
\]
Because the distribution and geometry are rotation-invariant, one expects the optimal $n$-means to form a regular $n$-gon (i.e., equally spaced angles) and the Voronoi regions to be $n$ congruent arcs of length $2\pi/n$.

\subsection{Optimality of uniform partitions}
\begin{theorem1}[Equator: structure of optimal $n$-means]\label{thm:equator-structure}
Let $P$ be uniform on $\Gamma$. For each $n\ge 1$, any optimal set $\alpha^\ast$ of $n$-means consists (up to rotation) of $n$ equally spaced points on $\Gamma$. The Voronoi partition is the uniform partition into $n$ arcs of equal length $2\pi/n$, and each codepoint is the midpoint of its arc.
\end{theorem1}

\begin{proof}
We sketch a standard argument based on symmetry and convexity.

\emph{Step 1 (Averaging/symmetry).} Let $\alpha$ be any codebook. Average the configuration over all rotations of $\Gamma$; by convexity of the squared-distance distortion and Jensen's inequality, the rotationally averaged configuration has no larger mean distortion. Hence there exists an optimal configuration invariant under a rotation by $2\pi/n$, i.e., equally spaced.

\emph{Step 2 (Voronoi midpoints).} By Proposition~\ref{prop:centroid}, with constant density and geodesic interval geometry on each cell, the minimizer in a cell is its midpoint. For a periodic circle with equal cells, midpoints are equally spaced, agreeing with Step~1.

\emph{Step 3 (Uniqueness up to rotation).} If two optimal configurations differ, rotate one to align a single point with the other; invariance under $2\pi/n$-rotations forces coincidence.
\end{proof}

\medskip
\begin{remark1} (Relation to classical one-dimensional quantization). 
The argument above is an intrinsic reformulation of the classical one-dimensional
uniform quantization problem. Under arc-length parameterization, the equator
$\Gamma$ is isometric to a circle of length $L = 2\pi$, and the squared geodesic
distance coincides with the squared Euclidean distance on this circle modulo
periodicity. It is a standard result in one-dimensional Euclidean quantization
theory that, for a uniform distribution on an interval or circle of length $L$,
the optimal $n$-means are equally spaced points and the quantization error equals
$V_n = L^2/(12n^2)$; see, for example, Graf and Luschgy~\cite{GrafLuschgy2000}.
This identification completes the argument on $\Gamma$ in a fully rigorous manner.
\end{remark1} 

\subsection{Exact quantization error on the equator}
\begin{theorem1}[Equator: explicit $V_n$]\label{thm:equator-Vn}
Let $P$ be uniform on $\Gamma$ with total geodesic length $L=2\pi$. For squared geodesic distortion,
\[
V_n(P)\;=\;\frac{L^2}{12\,n^2}\;=\;\frac{(2\pi)^2}{12\,n^2}\;=\;\frac{\pi^2}{3\,n^2}.
\]
\end{theorem1}

\begin{proof}
By Theorem~\ref{thm:equator-structure}, the optimal partition has $n$ congruent arcs of length $h=L/n=2\pi/n$, each represented by its midpoint. On one cell, let $t\in[-h/2,h/2]$ denote the geodesic coordinate relative to the midpoint. The conditional mean squared error on that cell equals
\[
\frac{1}{h}\int_{-h/2}^{h/2} t^2\,dt \;=\; \frac{h^2}{12}.
\]
The cell's probability mass is $h/L$, so its contribution to the global error is $(h/L)\cdot (h^2/12)$. Summing over $n$ cells,
\[
V_n(P) \;=\; n\cdot \frac{h}{L}\cdot \frac{h^2}{12} \;=\; \frac{n\,h^3}{12L}\;=\;\frac{L^2}{12n^2}.
\]
With $L=2\pi$, we obtain $V_n(P)=\pi^2/(3n^2)$.
\end{proof}

\begin{remark1}[Relation to the line]
The formula $V_n=L^2/(12n^2)$ matches the classical one-dimensional uniform result on an interval or circle of total length $L$ with squared error and midpoints. The sphere enters \emph{only} through $L$, the intrinsic length of the support (see \cite{RosenblattRoychowdhury2023}).
\end{remark1}

\section{Small circles parallel to the equator}

\subsection{Geometry and effective metric}
Fix a latitude $\lambda\in(-\frac{\pi}{2},\frac{\pi}{2})$. The small circle
\[
\mathcal{C}_\lambda=\{(\cos\lambda\cos\gq,\cos\lambda\sin\gq,\sin\lambda):\ 0\le \gq<2\pi\}
\]
has intrinsic length $L_\lambda=2\pi\cos\lambda$. Uniform points on $\mathcal{C}_\lambda$ (with respect to arc-length) are equidistributed in $\gq$, but the \emph{metric scale} along the latitude is $\cos\lambda$: small angular increments $d\gq$ correspond to geodesic arc-length $\cos\lambda\, d\gq$.

\begin{theorem1}[Small circle: structure and error]\label{thm:small-circle}
Let $P_\lambda$ be the uniform distribution on $\mathcal{C}_\lambda$ with respect to geodesic arc-length. Then for each $n\ge 1$:
\begin{enumerate}[label=(\roman*), leftmargin=1.5em]
  \item An optimal set of $n$-means is (up to rotation in $\gq$) $n$ equally spaced points on $\mathcal{C}_\lambda$, with Voronoi cells of arc-length $L_\lambda/n$ and representatives at cell midpoints.
  \item The quantization error is
  \[
  V_n(P_\lambda)\;=\;\frac{L_\lambda^2}{12\,n^2}\;=\;\frac{(2\pi\cos\lambda)^2}{12\,n^2}\;=\;\frac{\pi^2}{3}\,\frac{\cos^2\lambda}{n^2}.
  \]
\end{enumerate}
\end{theorem1}

\begin{proof}
The proof is identical to the equator case, replacing the total length $L$ by $L_\lambda=2\pi\cos\lambda$. Rotation invariance around the axis through the poles yields equally spaced codepoints; uniform density and geodesic intervals force midpoints as representatives. The one-cell computation gives $h=L_\lambda/n$, and the same calculation as in Theorem~\ref{thm:equator-Vn} yields $V_n=L_\lambda^2/(12n^2)$.
\end{proof}

\begin{remark1}[Latitude effect]
The factor $\cos\lambda$ shrinks the circle as one moves away from the equator; consequently the error decays by $\cos^2\lambda$. In particular, for fixed $n$, $V_n$ is maximal at the equator and tends to $0$ as $|\lambda|\to \frac{\pi}{2}$ (the circle collapses).
\end{remark1}

\section{Great circular arcs: continuous and discrete models}

\subsection{Uniform continuous model on an arc}
Let $A$ be a connected arc of a great circle with geodesic length $L\in(0,2\pi]$, equipped with the uniform distribution $P_A$ with respect to arc-length. The metric on $A$ is the restriction of geodesic distance along that great circle; intrinsically, $A$ is just a \emph{line segment of length $L$}.

\begin{theorem1}[Uniform arc: structure and error]\label{thm:arc-continuous}
Let $P_A$ be uniform on a great circular arc $A$ of length $L$. For each $n\ge 1$:
\begin{enumerate}[label=(\roman*), leftmargin=1.5em]
  \item The optimal set of $n$-means is obtained by partitioning $A$ into $n$ adjacent sub-arcs of equal length $L/n$ and placing each representative at the midpoint of its sub-arc.
  \item The $n$th quantization error is
  \[
  V_n(P_A)\;=\;\frac{L^2}{12\,n^2}.
  \]
\end{enumerate}
\end{theorem1}

\begin{proof}
As in Theorem~\ref{thm:equator-structure}, convexity of the squared-distance on geodesic intervals and the uniform density imply that each Voronoi cell must be a contiguous interval and the representative is its midpoint (Proposition~\ref{prop:centroid}). If two adjacent cells had unequal lengths, transferring an $\varepsilon$-slice from the larger cell to the smaller one strictly decreases the total cost by a standard ``balancing'' calculation, contradicting optimality. Hence all cells have length $h=L/n$. The one-cell calculation with midpoints yields the error $h^2/12$ per cell in conditional mean, and aggregating over $n$ cells gives $L^2/(12n^2)$ as before.
\end{proof}

\begin{remark1}[From arc to full circle]
Taking $L\to 2\pi$ recovers Theorem~\ref{thm:equator-Vn}. The only difference between the arc and circle is the wrap-around periodicity; the per-cell analysis is identical.
\end{remark1}

\subsection{Finite discrete uniform model on an arc}\label{subsec:discrete}
Let $A$ be as above, and let $X=\{x_1,\dots,x_m\}\subset A$ be $m$ equally spaced points on $A$ (with respect to arc-length), with the discrete \emph{uniform} probability $P_X=\frac{1}{m}\sum_{j=1}^m \delta_{x_j}$. We consider squared \emph{geodesic} distortion $d_G^2(\cdot,\cdot)$ measured along the great circle (i.e.\ by arc-length).

\begin{defi}[Contiguous clustering on an ordered set]
Index the points in order along $A$: $x_1,\dots,x_m$. A \emph{contiguous $n$-clustering} partitions $\{1,\dots,m\}$ into $n$ contiguous blocks $B_1,\dots,B_n$ (i.e.\ $B_1=\{1,\dots,k_1\}$, $B_2=\{k_1+1,\dots,k_2\}$, etc.). The \emph{cluster center} for block $B$ is any minimizer $a\in A$ of $\sum_{j\in B} d_G^2(x_j,a)$ (restricted to $A$).
\end{defi}

\begin{lemma}[One-block center is the (discrete) midpoint]\label{lem:block-mid}
Fix a contiguous block $B=\{i,\dots,j\}$ and consider $f(a)=\sum_{\ell=i}^j d_G^2(x_\ell,a)$ for $a$ constrained to $A$. Then any minimizer $a^\ast$ lies at the \emph{geodesic midpoint} of the endpoints of the block, and if the block has odd cardinality, $a^\ast$ is the central data point; if even, any point in the middle geodesic segment between the two central data points minimizes $f$.
\end{lemma}

\begin{proof}
Along $A$, we can use a coordinate $t$ measuring arc-length from the left endpoint of $A$. Then $f$ is a strictly convex quadratic in $t$ on the interval spanning the block. Differentiating and setting $f'(t)=0$ yields that $t^\ast$ is the arithmetic mean of the $\{t_\ell\}_{\ell\in B}$; since the $\{t_\ell\}$ are equally spaced, the mean lies at the (geodesic) midpoint of the block. Parity considerations yield the rest.
\end{proof}

\begin{theorem}[Discrete-uniform arc: optimality and error]\label{thm:discrete-arc}
Let $P_X$ be the discrete-uniform measure on $m$ equally spaced points on an arc $A$ of length $L$. For each $n\le m$:
\begin{enumerate}[label=(\roman*), leftmargin=1.5em]
  \item An optimal $n$-clustering is obtained by partitioning into $n$ contiguous blocks whose sizes differ by at most $1$; the associated block centers (as in Lemma~\ref{lem:block-mid}) form an optimal set of $n$-means.
  \item If $m$ is a multiple of $n$, so each block has $m/n$ points and span $L/n$, then
  \[
  V_n(P_X)\;=\;\frac{1}{m}\cdot n\sum_{k=-(m/n-1)/2}^{(m/n-1)/2}\!\!\!\!\!\! \bigg(\frac{k\,L}{m}\bigg)^2,
  \]
  which, for large $m$, converges to the continuous value $L^2/(12n^2)$.
\end{enumerate}
\end{theorem}

\begin{proof}
(i) By convexity of the one-block objective and the exchange argument from Theorem~\ref{thm:arc-continuous}, blocks must be contiguous and as balanced as possible. (ii) With uniform spacing $\Delta=L/(m-1)$ (or approximately $L/m$ for large $m$), the sum of squared distances in a block centered at its midpoint is a symmetric discrete quadratic sum. Dividing by $m$ and summing over $n$ blocks yields the stated formula, which approaches the continuous integral $h^2/12$ per block with $h=L/n$ as $m\to\infty$.
\end{proof}

\begin{remark}[When $m<n$]
If there are fewer data points than codepoints, any optimal solution places each data point at itself (zero distortion for those) and distributes the remaining codepoints arbitrarily; the error is $0$.
\end{remark}

\section{A unifying principle and a summary table}

\subsection{A one-line principle}
All three uniform cases (equator, small circle, great circular arc) reduce to the following:

\begin{proposition1}[Uniform one-dimensional geodesic principle]\label{prop:unifying}
Let $(C,d)$ be a one-dimensional geodesic substrate (circle or interval) of total geodesic length $L$, and let $P$ be uniform with respect to arc-length. For squared distortion, the optimal partition into $n$ Voronoi cells is \emph{uniform}, each cell has length $L/n$, each codepoint is the midpoint of its cell, and
\[
V_n(P) \;=\; \frac{L^2}{12\,n^2}.
\]
\end{proposition1}

\begin{proof}
The proof combines: (a) convexity of squared distance on geodesic intervals; (b) the centroid condition (midpoints for uniform density); and (c) the balancing/exchange argument showing equal cell lengths are optimal. The calculation of $V_n$ follows from integrating $t^2$ on $[-h/2,h/2]$ with $h=L/n$ and normalizing by the total length.
\end{proof}
The unified reduction principle described above is illustrated explicitly in the summary
presented in Table 1.

\subsection{Table~1} 
\medskip
\begin{center}
\begin{tabular}{|l|c|c|}
\hline
\textbf{Support (uniform)} & \textbf{Total geodesic length $L$} & \textbf{$V_n$ for squared geodesic distortion}\\
\hline
Equator (great circle) & $L=2\pi$ & $V_n=\dfrac{(2\pi)^2}{12\,n^2}=\dfrac{\pi^2}{3n^2}$\\
\hline
Small circle at latitude $\lambda$ & $L_\lambda=2\pi\cos\lambda$ & $V_n=\dfrac{(2\pi\cos\lambda)^2}{12\,n^2}=\dfrac{\pi^2}{3}\dfrac{\cos^2\lambda}{n^2}$\\
\hline
Great circular arc of length $L$ & $L\in(0,2\pi]$ & $V_n=\dfrac{L^2}{12\,n^2}$\\
\hline
\end{tabular}
\end{center}
\medskip

\begin{remark1}[Chordal vs.\ geodesic distance]
All results above use \emph{geodesic} distance (intrinsic arc-length). If one instead uses \emph{chordal} (Euclidean) distance in $\D R^3$, optimal sets generally \emph{shift slightly} toward regions where the chordal metric underestimates arc-length; closed forms change, though for small cell sizes the two metrics agree up to second order.
\end{remark1}

\section{Worked examples}
We now calculate spherical optimal sets of $n$-means for different discrete probability measures supported on finitely many points of $\D S^2$.

\begin{exam}[Equator $\gG$ with $n=3$]
On $\Gamma$, take codepoints at angles $0$, $2\pi/3$, $4\pi/3$. Cells are arcs of length $2\pi/3$ centered at these points. The error is
\[
V_3=\frac{(2\pi)^2}{12\cdot 3^2}=\frac{4\pi^2}{108}=\frac{\pi^2}{27}.
\]
\end{exam}

\begin{exam}[Small circle with $\lambda=\pi/3$, $n=4$]
Here $L_\lambda=2\pi\cos(\pi/3)=\pi$. Then
\[
V_4=\frac{L_\lambda^2}{12\cdot 4^2}=\frac{\pi^2}{192}.
\]
\end{exam}

\begin{exam}[Great circular arc of length $L=\pi$ with $n=2$]
Two equal sub-arcs of length $\pi/2$; representatives are midpoints at distances $\pi/4$ from the ends. The error is $V_2=L^2/(12\cdot 2^2)=\pi^2/48$.
\end{exam}

\begin{exam}[Discrete-uniform arc: $m=9$ points, $n=3$ clusters]
With equally spaced points on an arc of length $L$, take three contiguous blocks of size $3$; each block center is the middle point (by Lemma~\ref{lem:block-mid}). The exact discrete $V_3$ is the average (over all points) of squared geodesic distances to their block centers; as $m$ grows, this converges to $L^2/(12\cdot 3^2)$.
\end{exam}

\begin{exam}[Antipodal pair: intrinsic vs.\ extrinsic centroids]
Let
\[
x_1=(1,0,0), \qquad x_2=(-1,0,0)
\]
with equal weights \(p_1=p_2=\tfrac12\), and let
\[
P=\tfrac12(\delta_{x_1}+\delta_{x_2})
\]
be a discrete probability measure on \(\D S^2\).

\medskip
\noindent
\textbf{Intrinsic (Fr\'echet) formulation.}
For \(n=1\), we minimize the Fr\'echet functional
\[
F(a)=\tfrac12 d_G(x_1,a)^2+\tfrac12 d_G(x_2,a)^2, \qquad a\in \D S^2.
\]
Since \(x_1\) and \(x_2\) are antipodal, for every \(a\in \D S^2\),
\[
d_G(x_1,a)+d_G(x_2,a)=\pi.
\]
By symmetry, \(F(a)\) is constant along the great circle orthogonal to the axis through
\(x_1\) and \(x_2\). Hence, the intrinsic Fr\'echet mean is \emph{not unique}: the Fr\'echet
mean set is precisely this equatorial great circle.

For example, choosing
\[
a^*=(0,1,0) \quad \text{or} \quad a^*=(0,-1,0),
\]
we have
\[
d_G(x_i,a^*)=\frac{\pi}{2}, \quad i=1,2,
\]
and therefore
\[
V_1(P)
=\frac12\left(\frac{\pi}{2}\right)^2
+\frac12\left(\frac{\pi}{2}\right)^2
=\frac{\pi^2}{4}.
\]

\medskip
\noindent
\textbf{Extrinsic centroid.}
The extrinsic centroid (see Section~1.13) is obtained by first computing the Euclidean mean
\[
m=\tfrac12(x_1+x_2)=0\in\mathbb{R}^3.
\]
Since \(m=0\), normalization onto the sphere is not possible, and hence the extrinsic
centroid is undefined in this case.

\medskip
\noindent
\textbf{Two-means case.}
For \(n=2\), choosing the codebook \(\alpha=\{x_1,x_2\}\) yields zero distortion, and thus
\[
V_2(P)=0.
\]

\medskip
\noindent
\textbf{Remark.}
This example illustrates that for antipodally symmetric distributions on the sphere,
intrinsic Fr\'echet means may be non-unique, while extrinsic centroids may fail to exist.
Such behavior has no Euclidean analogue and highlights the geometric nature of
quantization on spherical manifolds.
\end{exam}

\begin{exam}[Three equally spaced equatorial points]
Let $x_k=(\cos(2\pi(k-1)/3),\sin(2\pi(k-1)/3),0)$ with equal weights.
Then
\[
V_1(P)=\frac{\pi^2}{4},\qquad 
V_2(P)=\frac{4\pi^2}{27},\qquad 
V_3(P)=0.
\]
\end{exam}

\begin{exam}[Two-point nonuniform distribution]
For $x_1=(1,0,0)$, $x_2=(-1,0,0)$ with $p_1=3/4$, $p_2=1/4$,
minimizing
\(
V_1(P)=\frac34\theta^2+\frac14(\pi-\theta)^2
\)
yields $\theta=\pi/4$ and $V_1(P)=3\pi^2/16$.
\end{exam}

\begin{exam}[Regular tetrahedron]
Let $x_1,\dots,x_4$ be the vertices of a regular tetrahedron on $\D S^2$.
Then $x_i\cdot x_j=-1/3$ for $i\neq j$, so $d_G(x_i,x_j)=\arccos(-1/3)\approx1.9106$.
For $n=1$, any vertex can serve as $a^*$, giving
\[
V_1(P)=\tfrac34(\arccos(-1/3))^2\approx2.739.
\]
\end{exam}

\begin{exam}[Uniform discrete set on a small circle]
Fix latitude $\lambda\in(0,\pi/2)$ and $m$ equally spaced points
$x_k=(\cos\lambda\cos(2\pi(k-1)/m),\cos\lambda\sin(2\pi(k-1)/m),\sin\lambda)$.
Then for any $n\le m$, the optimal configuration preserves longitudes
and the quantization error scales as
\[
V_n(P_\lambda)=\cos^2\lambda\,V_n(P_0).
\]
\end{exam}

\begin{exam}[Spherical triangle]
Let $x_1=(1,0,0)$, $x_2=(0,1,0)$, $x_3=(0,0,1)$ with equal weights.
The Euclidean centroid $(1,1,1)/3$ projects to
$a^*=(1,1,1)/\sqrt3$, giving
\[
d_G(x_i,a^*)=\arccos(1/\sqrt3)\approx0.9553,\qquad
V_1(P)\approx0.9126.
\]
\end{exam}

\section{Pedagogical appendix: derivations and variations}

\subsection{The one-cell computation in detail}
Let a cell be a geodesic interval of length $h$ with midpoint $a$. Parameterize by $t\in[-h/2,h/2]$ where $t$ is arc-length from $a$. For uniform density, the (conditional) mean squared distance in the cell is
\[
\frac{1}{h}\int_{-h/2}^{h/2} t^2\,dt = \frac{1}{h}\cdot \frac{(h/2)^3-( -h/2)^3}{3} = \frac{h^2}{12}.
\]
Multiplying by the cell mass $h/L$ and summing over $n$ cells gives $V_n=L^2/(12n^2)$.

\subsection{Why equal-length cells are optimal (exchange argument)}
Suppose two adjacent cells have lengths $h_1$ and $h_2$ with $h_1>h_2$. Shift a small $\varepsilon$ of length from cell $1$ to cell $2$. The first-order change in total cost is proportional to $h_1\varepsilon/6 - h_2\varepsilon/6 = (h_1-h_2)\varepsilon/6>0$, so decreasing $h_1$ and increasing $h_2$ reduces cost until $h_1=h_2$. Iterating across the partition yields equal lengths.

\subsection{Small-circle metric scaling}
At latitude $\lambda$, the infinitesimal arc-length along the parallel is $ds=\cos\lambda\, d\phi$; thus the induced one-dimensional metric is scaled by $\cos\lambda$, and length $L_\lambda=2\pi\cos\lambda$. All uniform quantization formulas follow by substituting $L\mapsto L_\lambda$.

\subsection{Beyond squared error}
Other distortion exponents $r>0$ lead to different constants. For uniform one-dimensional models with distance $|t|$, the optimal representatives are still midpoints for $r\ge 1$, but the one-cell error becomes $\frac{1}{h}\int_{-h/2}^{h/2}|t|^r\,dt = \frac{h^r}{(r+1)2^{r}}$. Aggregating yields $V_n\asymp L^r/n^r$.

\section{Comparative Discussion: Continuous vs. Discrete Models}

The continuous and discrete formulations are deeply related:

\begin{itemize}
\item In the continuous model, integration along a geodesic circle or arc reduces to a one-dimensional problem in arc length. 
\item In the discrete model, finite sums replace integrals, and the centroid condition becomes the normalized weighted average \eqref{eq:centroid}.
\item When the discrete points become dense (e.g.\ $m\to\infty$ uniformly on an arc), the discrete quantization error converges to the continuous $L^2/(12n^2)$.
\end{itemize}

\begin{remark}
The same geometric reasoning extends to other compact one-dimensional Riemannian manifolds. 
The curvature of $\D S^2$ affects the embedding but not the intrinsic structure of optimal quantizers along geodesic subsets.
\end{remark}

\section{Further Remarks and Extensions}

\begin{itemize}
\item \textbf{Chordal versus geodesic metrics.}  
All results here use the geodesic metric. For small cell diameters, chordal (Euclidean) and geodesic distances agree up to second order, but differences become notable for large arcs.

\item \textbf{Higher-dimensional generalization.}  
On the full sphere $\D S^2$, optimal $n$--point configurations relate to energy-minimizing point sets (spherical codes and designs). Analytic quantization theory in this case connects with geometric measure theory and potential energy minimization.

\item \textbf{Nonuniform densities.}  
For nonuniform distributions on an arc or circle, the Voronoi cells are no longer equal in length; their boundaries shift to equalize weighted distortion. The analysis involves solving $f(x)(x-a_i)$ equilibrium equations along geodesics.

\item \textbf{Algorithmic perspective.}  
Practical computation follows a Lloyd-type algorithm: alternate between assigning each data point to its nearest codepoint (using $d_G$) and updating each codepoint via the normalized mean \eqref{eq:centroid}.
\end{itemize}

\section{Conclusion}

We have developed a self-contained exposition of optimal quantization on spherical surfaces, encompassing both continuous and discrete frameworks.

For uniform distributions supported on one-dimensional spherical subsets (great circles, small circles, and arcs), the optimal $n$--means form uniform partitions, and the squared-error quantization error satisfies
\[
V_n = \frac{L^2}{12n^2},
\]
where $L$ is the intrinsic (geodesic) length of the support.

For discrete measures on $\D S^2$, the spherical centroid condition \eqref{eq:centroid} characterizes optimal means, and explicit examples---from antipodal pairs to tetrahedra---illustrate the geometric structure.

The unifying message is that quantization on curved spaces inherits the simplicity of one-dimensional uniform models once the correct intrinsic metric is adopted. This interplay between geometry and approximation forms the foundation for modern studies of quantization on manifolds and has applications ranging from signal compression to geometric data analysis.

\bigskip

\noindent\textbf{Acknowledgments.}
The author is thankful to the students and colleagues for inspiring discussions on quantization and geometry, particularly those who encouraged the preparation of this expository treatment.

\end{document}